\documentclass[12pt,reqno]{article}

\usepackage{xcolor}
\usepackage{amssymb}
\usepackage{amsmath}
\usepackage{amsthm}
\usepackage{amsfonts} 
\usepackage{amscd}
\usepackage{graphicx}

\usepackage[colorlinks=true,
linkcolor=webgreen,
filecolor=webbrown,
citecolor=webgreen]{hyperref}

\definecolor{webgreen}{rgb}{0,.5,0}
\definecolor{webbrown}{rgb}{.6,0,0}

\usepackage{color}
\usepackage{fullpage}
\usepackage{float}

\usepackage{latexsym}
\usepackage{accents}

\DeclareMathOperator{\Li}{Li}

\begin{document}

\theoremstyle{plain}
\newtheorem{theorem}{Theorem}
\newtheorem{corollary}[theorem]{Corollary}
\newtheorem{lemma}{Lemma}
\newtheorem{example}{Example}
\newtheorem{remark}{Remark}

\begin{center}
\vskip 1cm{\LARGE\bf 
Combinatorial sums, series and integrals involving odd harmonic numbers \\
}
\vskip 1.cm
{\large
Kunle Adegoke \\
Department of Physics and Engineering Physics \\ Obafemi Awolowo University, Ile-Ife\\Nigeria \\
\href{mailto:adegoke00@gmail.com}{\tt adegoke00@gmail.com}

\vskip 0.22 in

Robert Frontczak \\
Independent Researcher\\ Reutlingen \\ Germany \\
\href{mailto:robert.frontczak@web.de}{\tt robert.frontczak@web.de}

\vskip 0.22 in

Taras Goy  \\
Faculty of Mathematics and Computer Science\\
Vasyl Stefanyk Precarpathian National University, Ivano-Frankivsk \\Ukraine\\
\href{mailto:taras.goy@pnu.edu.ua}{\tt taras.goy@pnu.edu.ua}}
\end{center}

\vskip .22 in

\begin{abstract}
We present several types of ordinary generating functions involving central binomial coefficients, harmonic numbers, and odd harmonic numbers. Our results complement those of Boyadzhiev from 2012 and Chen from 2016. Based on these generating functions we evaluate several infinite series in closed form. In addition,   we offer some combinatorial sum identities involving Catalan numbers, harmonic numbers and odd harmonic numbers. Finally, we analyze a special log-integral with Fibonacci numbers and odd harmonic numbers.
\vskip .1 in

\end{abstract}

\section{Introduction and motivation}

Let $\Gamma(z)$ be the familiar gamma function given by 
$\Gamma(z) = \int\limits_0^\infty e^{-t} t^{z-1} dt$, $\Re(z)>0$. 
The digamma function $\psi(z)$ is defined for all $z\in\mathbb{C}\setminus \{0,-1,-2,\ldots \}$ by
\begin{equation*}
\psi(z) = 
\frac{\Gamma'(z)}{\Gamma(z)}
= - \gamma - \frac{1}{z} + \sum_{k=1}^\infty \Big (\frac{1}{k}-\frac{1}{k+z}\Big ),
\end{equation*}
with $\gamma$ being the Euler--Mascheroni constant
$\gamma = \lim\limits_{n\rightarrow \infty}\Big ( \sum\limits_{k=1}^n\frac{1}{k} - \ln n\Big) = 0,577215\ldots.$

We use the notation $O_n$ for the $n$th odd harmonic number,
$O_n = \sum\limits_{j=1}^n \frac{1}{2j-1}$, $O_0 = 0$, by analogy with the usual notation $H_n$ for the classical harmonic numbers
$H_n = \sum\limits_{j=1}^n \frac{1}{j}$, $H_0 = 0$.
Since
\begin{equation*}
H_n  = \sum_{j = 1}^n {\frac{1}{j}} = \sum_{j = 1}^{\left\lfloor {n/2} \right\rfloor } {\frac{1}{{2j}}} + \sum_{j = 1}^{\left\lceil {n/2} \right\rceil } {\frac{1}{{2j - 1}}} = \frac{1}{2}H_{\left\lfloor {n/2} \right\rfloor } + O_{\left\lceil {n/2} \right\rceil },
\end{equation*}
odd harmonic numbers and harmonic numbers are related by
$H_{2n} = \frac{1}{2} H_n + O_n$ and
\begin{equation}\label{eq.lreuv6s}
H_{2n - 1} = \frac{1}{2}H_{n - 1} + O_n.
\end{equation}

Let $F_n$ denote the $n$-th Fibonacci and $L_n$ the $n$-th Lucas number, both satisfying the recurrence
$u_n = u_{n-1}+u_{n-2}$, $n\geq2$,	
but with respective initial conditions 
$F_0=0$, $F_1=1$ and
$L_0=2$, $L_1=1$. For negative subscripts we have 
$F_{-n}=(-1)^{n-1}F_n$ and $L_{-n}=(-1)^n L_n$. 

The Binet formulas for these numbers state that
\begin{equation}\label{bine}
F_n = \frac{\alpha^n - \beta^n }{\alpha - \beta },\quad L_n = \alpha^n + \beta^n,\quad n\in\mathbb Z,
\end{equation}
where $\alpha=(1 + \sqrt 5)/2$ is the golden ratio and $\beta=-1/\alpha$.

For further information
on Fibonacci and Lucas numbers, we refer the reader to entries  A000045 and A000032, respectively, in the On-Line Encyclopedia of Integer Sequences \cite{OEIS}.

Let $(G_j(a,b))_{j\in\mathbb Z}$ be the \textit{gibonacci sequence} having the same recurrence relation as the Fibonacci and Lucas sequences but starting with arbitrary initial values; that is, let $G_0  = a$, $G_1  = b$, $G_j(a,b)  = G_{j - 1} (a,b) + G_{j - 2}(a,b)$, $j \ge 2$, where $a$ and $b$ are arbitrary numbers (usually integers) not both zero. Thus, $F_j=G_j(0,1)$ and $L_j=G_j(2,1)$. For brevity, we will  write $G_j$ for $G_j(a,b)$. Extension to negative subscripts is provided by writing the recurrence  as $G_{-j}=G_{-(j - 2)} - G_{-(j - 1)}$; so that the gibonacci sequence is defined for all integers. The gibonacci numbers can be accessed directly through the Binet-like formula
\begin{equation}\label{binet-gibo}
G_j  = \frac{{(b - a\beta )\alpha ^j  + (a\alpha  - b)\beta^j }}{{\alpha  - \beta }}.
\end{equation}

Finally, central binomial coefficients and Catalan numbers $C_n$ are defined for all $n\geq 0$ by $\binom {2n}{n} = \frac{(2n)!}{(n!)^2}$ and $C_n = \frac{1}{n+1} \binom {2n}{n}$.

This paper is motivated by the papers by Boyadzhiev from 2012 \cite{Boyadzhiev} and Chen from 2016 \cite{Chen}. Boyadzhiev studied generating function involving central binomial coefficients and harmonic numbers $H_n$. His main results are \cite[Theorem 1]{Boyadzhiev}
\begin{align}\label{eq.royfy2s}
\sum_{n=0}^\infty \binom {2n}{n} H_n x^n &= \frac{2}{\sqrt{1-4x}} \ln\Big (\frac{1+\sqrt{1-4x}}{2\sqrt{1-4x}}\Big ),\quad x\in\Big[\!-\frac14;\frac14\Big),\\
\sum_{n=0}^\infty (-1)^{n+1}\binom {2n}{n}  H_n x^n &= \frac{2}{\sqrt{1+4x}} \ln\Big (\frac{2\sqrt{1+4x}}{1+\sqrt{1+4x}}\Big ),\quad  x\in\Big(\!-\frac14;\frac14\Big].\nonumber
\end{align}

Chen added results for the sequences $\binom {2n}{n}(H_{2n} - H_n)$, $C_n (H_{2n} - H_n)$ and $\binom{2n}{n}O_n$,
and proved among others \cite[Theorem 8]{Chen}
\begin{equation}\label{chen_id}
\sum_{n=0}^\infty \binom{2n}{n} O_n x^n = - \frac{\ln\sqrt{1-4x}}{\sqrt{1-4x}},\quad x\in\Big[\!-\frac14;\frac14\Big).
\end{equation}

Here we derive generating functions for several related sequences, starting with a different proof of Chen's identity \eqref{chen_id} using a special family of log-integrals. As a consequence of \eqref{chen_id} we state and prove interesting combinatorial sum identities involving Catalan numbers, harmonic numbers, and odd harmonic numbers. Proceeding further we offer generating functions for
$\sum\limits_{n=0}^\infty \binom {2n}{n} H_{2n+1} x^n$, $\sum\limits_{n=0}^\infty \binom {2n}{n} \frac{O_n}{n+m+1} x^n$ with $m\geq 0$, and others. Our results are also closely related to other articles, among others we mention Lehmer \cite{Lehmer},
Chu and Zheng \cite{Chu}, Chen \cite{Chen2}, Furdui and S\^int\u am\u arian \cite{Furdui}, and Stewart \cite{Stewart}, who found various ordinary generating functions for sequences involving products between harmonic numbers and Fibonacci (Lucas) numbers.
In the final section, we provide an analysis of a special log-integral with Fibonacci numbers and odd harmonic numbers.

\section{An integral and some consequences}

The next lemma will be crucial for several parts of this paper.
\begin{lemma}\label{fund_lem}
	For all real $a>0$ and integers $n\geq 0$,
	\begin{equation}\label{lem_eq}
	\int_0^\infty \frac{\ln x}{(a^2 + x^2)^{n+1}}\, dx = \pi\binom {2n}{n} \frac{\ln a - O_n}{(2a)^{2n+1}}.
	\end{equation}
\end{lemma}
\begin{proof}
	Entry 4.231\,(7) in Gradshteyn and Ryzhik \cite{GrRy07} states that for $a>0$
	\begin{equation*}
	\int_0^\infty \frac{\ln x}{(a^2 + x^2)^{n}}\, dx
	= \frac{\sqrt{\pi}\,\Gamma\big(n-\frac{1}{2}\big)}{4 (n-1)!\, a^{2n-1}} \left( 2\ln \Big (\frac{a}{2}\Big ) - \gamma - \psi\Big(n-\frac{1}{2}\Big )\right).
	\end{equation*}
	Replace $n$ by $n+1$ and use
	$\Gamma\big(n+\frac{1}{2}\big) = \frac{(2n)!\sqrt{\pi}}{4^n n!}$ and $\psi\big (n+\frac{1}{2}\big)= - \gamma - 2\ln 2 + 2\,O_n$.
\end{proof}
\begin{lemma}\label{lem2}
	For all $a\geq1$, 
	\begin{equation}\label{2.1}
	\sum_{n=0}^\infty (-1)^n \binom {2n}{n} \frac{\ln a - O_n}{(2a)^{2n}}  = \frac{a\ln(a^2+1)}{2\sqrt{a^2+1}},
	\end{equation}
	and, for all $a>1$,
	\begin{equation}
	\label{2.2}
	\sum_{n=0}^\infty \binom {2n}{n} \frac{\ln a - O_n}{(2a)^{2n}}  = \frac{a\ln(a^2-1)}{2\sqrt{a^2-1}}.
	\end{equation}
\end{lemma}
\begin{proof}
	Use \eqref{lem_eq} and sum from $n=0$ to $\infty$ using the geometric series. Note that
	$\int\limits_0^\infty \frac{\ln x}{a^2 + x^2} dx = \frac{\pi}{2a} \ln a$.
\end{proof}

Special cases of \eqref{2.1} and \eqref{2.2} are
\begin{gather*}
\sum_{n=0}^\infty (-1)^{n} \binom {2n}{n} \frac{O_n}{4^{n}}  = - \frac{\ln 2}{2\sqrt{2}}, \qquad
\sum_{n=0}^\infty (-1)^{n} \binom {2n}{n} \frac{\ln 2 - O_n}{16^{n}}  = \frac{\ln 5}{\sqrt{5}},\\
\sum_{n=0}^\infty (-1)^{n} \binom {2n}{n} \frac{\ln\alpha - O_n}{(2\alpha)^{2n}}
= \frac{\alpha\ln(\sqrt{5}\alpha)}{2\sqrt{2+\alpha}},\qquad \sum_{n=0}^\infty (-1)^{n} \binom {2n}{n} \frac{\ln\alpha -2 O_n}{(4\alpha)^{n}}
= \frac{2\ln\alpha}{\sqrt{\alpha}},
\end{gather*}
and
\begin{gather*}
\sum_{n=0}^\infty \binom {2n}{n} \frac{\ln\alpha - O_n}{(2\alpha)^{2n}} = \frac{\sqrt{\alpha}}{2} \ln\alpha,\qquad \sum_{n=0}^\infty \binom {2n}{n} \frac{\ln2 - O_n}{16^{n}} = \frac{\ln3}{\sqrt3},\\
\sum_{n=0}^\infty \binom {2n}{n} \frac{\ln(2\alpha) - 2O_n}{(8\alpha)^{n}} = \frac{\sqrt{2\alpha}\ln5}{2\sqrt[4]5}.
\end{gather*}

Our first theorem rediscovers Chen's generating function \eqref{chen_id} \cite[Theorem 8]{Chen}.
\begin{theorem}
	We have
	\begin{align}\label{main_id1}
	\sum_{n=0}^\infty \binom {2n}{n} O_n x^n &=- \frac{\ln(1-4x)}{2\sqrt{1-4x}} ,\quad x\in\Big[\!-\frac14;\frac14\Big),\\
	\label{main_id2}
	\sum_{n=0}^\infty (-1)^{n} \binom {2n}{n}  O_n x^n &=- \frac{\ln(1+4x)}{2\sqrt{1+4x}},\quad x\in\Big(\!-\frac14;\frac14\Big].\nonumber
	\end{align}
\end{theorem}
\begin{proof}
	From Lehmer's paper \cite{Lehmer} (or by the generalized binomial theorem) we have
	\begin{equation*}
	\sum_{n=0}^\infty \binom {2n}{n} x^n = \frac{1}{\sqrt{1-4x}}, \quad x\in\Big(\!-\frac14;\frac{1}{4}\Big],
	\end{equation*}
	which shows that
	\begin{equation*}
	\sum_{n=0}^\infty \binom {2n}{n} \frac{1}{(2a)^{2n}} = \frac{a}{\sqrt{a^2-1}}.
	\end{equation*}
	Hence, from Lemma \ref{lem2}
	\begin{equation*}
	\sum_{n=0}^\infty \binom {2n}{n} \frac{O_n}{(2a)^{2n}}  = \frac{a \ln\big (\frac{a^2}{a^2-1}\big)}{2 \sqrt{a^2-1}},\quad a>1.
	\end{equation*}
	Set $x=(2a)^{-2}$ and simplify.
\end{proof}

We continue with several combinatorial identities involving harmonic numbers $H_n$, odd harmonic numbers $O_n$, and Catalan numbers $C_n$.
\begin{corollary}
	For integers $n\geq 0$, we have 
	\begin{equation}\label{cor_id1}
	\sum_{j=0}^n 4^{n-j} C_j H_{n-j}  = 2^{2n+1} H_{n+1} - \binom {2(n+1)}{n+1} O_{n+1}.
	\end{equation}
\end{corollary}
\begin{proof}
	Write the right-hand side of \eqref{main_id1} as
	$\frac{1}{2} \sqrt{1-4x}\cdot \left( - \frac{\ln (1-4x)}{1-4x} \right)$
	and use the gene\-rating functions
	\begin{equation*}
	\frac{1}{2} \sqrt{1-4x} = \frac{1}{2} \sum_{n=0}^\infty 4^n \Big (-\frac{1}{2}\Big )_n \frac{x^n}{n!}
	= \frac{1}{2} - \sum_{n=0}^\infty \binom {2n}{n} \frac{x^{n+1}}{n+1}
	\end{equation*}
	where  $(\lambda)_n=\lambda(\lambda + 1)\cdots (\lambda +n-1)$, $(\lambda)_0=1,$ is the Pochhammer symbol and
	$	- \frac{\ln (1-4x)}{1-4x} = \sum\limits_{n=0}^\infty 4^n H_n x^n.$
	
	Apply Cauchy's product rule, then extract and compare the coefficients of $x^n$.
\end{proof}
\begin{remark}
	We can write \eqref{cor_id1} equivalently as follows
	\begin{equation*}
	\sum_{j=0}^n 4^{j} C_{n-j} H_{j}  = 2^{2n+1} H_{n+1} - (n+2)C_{n+1}O_{n+1}.
	\end{equation*}
	Also, it follows from the known formula $C_n=\frac{n+2}{4n+2}\,C_{n+1}$ that $C_{-1}=-\frac12$.
	Hence, we have a still other form of \eqref{cor_id1} given by
	\begin{equation*}
	\sum_{j=0}^{n} 4^{j} C_{n-j-1} H_{j+1} = -\frac{n+2}{4}C_{n+1}O_{n+1}.
	\end{equation*}
\end{remark}
\begin{corollary}
	For integers $n\geq 0$, we have 
	\begin{equation*}
	\sum_{j=0}^n \binom {2(n-j)}{n-j}  C_j O_{n-j}= \frac{1}{2} \binom {2(n+1)}{n+1} O_{n+1} - \frac{4^{n}}{n+1}.
	\end{equation*}
\end{corollary}
\begin{proof}
	Work with the relation
	$$
	2 \sqrt{1-4x}\cdot \sum\limits_{n=0}^\infty \binom {2n}{n} O_n x^n = - \ln (1-4x)
	$$
	and apply Cauchy's product rule.
\end{proof}
\begin{corollary}
	For integers $n\geq 0$, we have
	\begin{equation*}
	\sum_{j=0}^n \binom {2j}{j} \binom {2(n-j)}{n-j} O_j O_{n-j} = 4^{n-1} \sum_{j=1}^n \frac{H_{n-j}}{j}.
	\end{equation*}
\end{corollary}
\begin{proof}
	Work with the relation
	$$	\Big ( \frac{\ln (1-4x)}{\sqrt{1-4x}} \Big )^2 = - \ln (1-4x)\cdot \left( - \frac{\ln (1-4x)}{1-4x} \right)
	$$
	in conjunction with the  series
	$$	- \ln (1-x) = \sum\limits_{n=1}^\infty \frac{x^n}{n},$$	
	and Cauchy's product rule.
\end{proof}
\begin{remark} We note that due to a result of Carlitz \cite{Carlitz} we also have for integers $r\geq 1$
	\begin{equation*}
	\sum_{j=0}^n \binom {2j}{j} \binom {2(n-j)}{n-j} O_j O_{n-j} F_{rj} L_{r(n-j)} = 4^{n-1} F_{rn}  \sum_{j=1}^n \frac{H_{n-j}}{j}.
	\end{equation*}
\end{remark}
\begin{corollary}
	For integers $n\geq 0$,
	\begin{equation*}
	\sum_{j=0}^n \binom {2j}{j} \frac{O_{j} - H_{n-j}}{4^{j}(n+1-j)}  = 0.
	\end{equation*}
\end{corollary}
\begin{proof}
	Work with the generating function \cite[Entry (7.1)]{Furdui2}
	\begin{equation*}
	\sum_{n=1}^\infty \frac{H_n}{n} x^n = \Li_2(x) + \frac{1}{2} \ln^2(1-x), \quad x\in[-1,1),
	\end{equation*}
	where $\Li_2(x)$ is the dilogarithm function defined, for $|x|\leq1$, by $\Li_2(x)=\sum\limits_{n=1}^{\infty}\frac{x^n}{n^2}$. Note that
	\begin{equation*}
	\sum_{n=1}^\infty \frac{H_n}{n} x^n - \Li_2(x) = \sum_{n=1}^\infty \frac{H_{n-1}}{n} x^n.
	\end{equation*}
	Multiply through both sides by $-1/\sqrt{1-4x}$, apply Cauchy's product rule, and extract the coefficient of $x$. The result is
	\begin{equation*}
	\sum_{j=1}^n \binom {2j}{j} \frac{O_j - H_{n-j}}{4^{j}(n+1-j)}  = \frac{H_{n}}{n+1},
	\end{equation*}
	which is the stated identity.
\end{proof}

In the next three theorems the attentive reader will recognize connections to hyperharmonic numbers (see \cite{Benjamin,Dil}).
\begin{theorem}\label{main_thm2}
	For integers $n$, $p\geq 0$, we have the following identity: 
	\begin{equation}\label{thm2_id}
	\begin{split}
	\sum_{j=0}^n \binom {p+j}{j} \binom {2(n-j)}{n-j} 4^{j} O_{n-j}  &= 2^{2n-1} \binom {n+p+1}{n} \left ( H_{n+p+1} - H_{p+1} \right ) \\
	& \quad - \sum_{j=1}^n \binom {p+j}{j-1} 4^{j-1}C_{n-j}  \left ( H_{p+j} - H_{p+1} \right ).
	\end{split}
	\end{equation}
	In particular,
	\begin{equation*}
	\sum_{j=0}^n \binom {2j}{j}  \frac{O_{j}}{4^j}  = \frac{n+1}{2} ( H_{n+1} - 1)
	- \sum_{j=1}^n  \frac{j\,C_{n-j}}{4^{n+1-j}}  \left ( H_{j} - 1 \right ).
	\end{equation*}
\end{theorem}
\begin{proof}
	We start with the generating function
	$$(1-x)^{-\lambda} = \sum\limits_{n=0}^\infty \frac{(\lambda)_n}{n!} x^n,\quad \lambda \in\mathbb{C}, \,\, |x|<1.
	$$
	As
	$\frac{d}{d\lambda} (\lambda)_n = \frac{d}{d\lambda} \frac{\Gamma(\lambda + n)}{\Gamma(\lambda)}
	= (\lambda)_n \big(\psi (\lambda + n) - \psi(\lambda) \big)$
	we get
	\begin{equation*}
	\frac{d}{d\lambda} (1-x)^{-\lambda} = - \frac{\ln (1-x)}{(1-x)^{\lambda}}
	= \sum_{n=0}^\infty \frac{(\lambda)_n}{n!} \big (\psi (\lambda + n) - \psi(\lambda) \big ) x^n.
	\end{equation*}
	Thus, upon replacing $\lambda$ by $p+1$, we get
	\begin{align*}
	\sum_{n=0}^\infty \binom {n+p}{n} &4^n x^n \cdot \sum_{n=0}^\infty \binom {2n}{n} O_n x^n 
	= \frac{1}{2} \sqrt{1-4x}\cdot  \left( - \frac{\ln (1-4x)}{(1-4x)^{p+2}} \right)
	\\
	\qquad &= \frac{1}{2} \sqrt{1-4x} \cdot  \sum_{n=0}^\infty \binom {n+p+1}{n} 4^n \big (\psi (p + 2 + n) - \psi(p+2) \big ) x^n
	\\
	\qquad &= \frac{1}{2} \sqrt{1-4x} \cdot \sum_{n=0}^\infty \binom {n+p+1}{n} 4^n \left (H_{p + 1 + n} - H_{p+1} \right ) x^n,
	\end{align*}
	where in the last step it was used that $\psi(n)=-\gamma + H_{n-1}$. Now, proceed as before, work with Cauchy's product rule, 	extract and compare the coefficients of $x^n$.
\end{proof}
\begin{remark}
	An equivalent version of \eqref{thm2_id} is given by
	\begin{equation*}
	\sum_{j=0}^n \binom{p+j}{j} \binom{2(n-j)}{n-j} 4^{j} O_{n-j} = -\sum_{j=0}^{n} \binom{p+j+1}{j} 4^j C_{n-j-1} (H_{p+j+1} - H_{p+1}),
	\end{equation*}
	where  $C_{-1}=-1/2$ was used.
\end{remark}

The next theorem offers a combinatorial expression for the same types of summands as in Theorem \ref{main_thm2}
but with an additional factor $1/j$.
\begin{theorem} 
	For all integers $n\geq 1$ and $p\geq 0$, we have the following identity:
	\begin{align*}
	&\sum_{j=1}^n  \binom {p+j}{j} \binom {2(n-j)}{n-j} \frac{4^{j}}{j} O_{n-j} = - \binom {2n}{n} H_p O_n + 2^{2n-1} S_n
	- \sum_{j=0}^{n-1}  4^{n-j-1} C_j S_{n-j-1} \\
	&\quad + \frac12\sum_{k=1}^p \frac{1}{k} \left ( 4^{n}\binom {n+k}{n} (H_{n+k}-H_k) -\frac12 \sum_{j=1}^{n}  \binom {j+k-1}{j-1} 4^{j}( H_{k+j-1} - H_{k})C_{n-j} \right )\!,
	\end{align*}
	where, for all $n\geq 1$, $S_n$ is defined by
	$$
	S_n = \sum\limits_{m=1}^n \frac{H_{n-m}}{m},\quad S_0 = 0.$$
	In particular,
	\begin{equation*}
	\sum_{j=1}^n \binom {2(n-j)}{n-j}  \frac{4^{j}}{j} O_{n-j} = 2^{2n-1} S_n - \sum_{j=1}^n 4^{j-1} C_{n-j}  S_{j-1}.
	\end{equation*}
\end{theorem}
\begin{proof}
	From \cite[Theorem 3]{Boyadzhiev2} it is known that
	\begin{equation*}
	\sum_{n=1}^\infty \binom {n+p}{n} \frac{x^n}{n} = - H_p - \ln(1-x) + \sum_{k=1}^p \frac{1}{k(1-x)^k}, \quad |x|<1.
	\end{equation*}
	Thus, applying Cauchy's product rule once more time and keeping in mind that $O_0=0$ we obtain
	\begin{align*}
	\sum_{n=1}^\infty &\binom {n+p}{n}  \frac{4^n}{n} x^n \cdot \sum_{n=1}^\infty \binom {2n}{n} O_n x^n= \sum_{n=1}^\infty \left ( \sum_{j=1}^n \binom {j+p}{j} \frac{4^j}{j} \binom {2(n-j)}{n-j} O_{n-j}\right ) x^n \\
	&= - \frac{H_p \sqrt{1-4x}}{2} \left( - \frac{\ln (1-4x)}{1-4x} \right) - \frac{\sqrt{1-4x}}{2} \left( - \frac{\ln^2 (1-4x)}{1-4x} \right)\\
	&\quad+ \sum_{k=1}^p \frac{\sqrt{1-4x} }{2k} \left( - \frac{\ln (1-4x)}{(1-4x)^{k+1}} \right)\!.
	\end{align*}
	Extract the coefficient of $x^n$ while making use of
	\begin{equation}\label{Sn_id}
	- \frac{\ln^2 (1-4x)}{1-4x} = - \sum_{n=1}^\infty 4^n S_n x^n.
	\end{equation}
\end{proof}
\begin{corollary} For $n\geq 1$,
	$$S_n = 2 \sum\limits_{m=1}^n \frac{H_{m-1}}{m}.
	$$
\end{corollary}
\begin{proof}
	This follows from combining the power series of $\ln^2 (1-4x)$ with the geometric series $(1-4x)^{-1}$, extracting the coefficient of $x^n$,
	and comparing with \eqref{Sn_id}. A second proof can be given by induction on $n$ using the elementary identity
	$	\sum\limits_{j=1}^{n-1} \frac{1}{j(n-j)} = \frac{2H_{n-1}}{n}$.
\end{proof}
\begin{theorem}
	For integers $n$, $p\geq 0$,
	\begin{align*}
	\sum_{j=0}^n \binom {j+p}{j} \binom {2(n-j)}{n-j} 4^{j} H_{j} O_{n-j}& = \sum_{j=0}^n \binom {j+p}{j} \binom  {2(n-j)}{n-j} 4^j O_{n-j} H_{j+p} \\
	&\quad - \sum_{k=1}^p \frac{1}{k} \sum_{j=0}^n \binom {j+p-k}{j} \binom {2(n-j)}{n-j} 4^{j} O_{n-j}.
	\end{align*}
	In particular,
	\begin{equation*}
	H_{n+p} = H_n + \frac{1}{\binom {n+p}{n}} \sum_{k=1}^p \frac{1}{k} \binom {n+p-k}{n}.
	\end{equation*}
\end{theorem}
\begin{proof}
	Work with the generating function \cite[Theorem 3]{Boyadzhiev2}
	\begin{equation}\label{eq.rqnwwad}
	\sum_{n=0}^\infty \binom {n+p}{n} H_n x^n = \frac{1}{(1-x)^{p+1}}\left ( H_p - \ln(1-x) - \sum_{k=1}^p \frac{(1-x)^k}{k}\right ), \quad |x|<1.
	\end{equation}
\end{proof}

Our final combinatorial identity involves harmonic numbers of second order $H_n^{(2)}$, defined by  $H_n^{(2)}=\sum\limits_{k=1}^n \frac{1}{k^2}$ with $H_0^{(2)}=0$.
\begin{corollary}
	For integers $n\geq 0$, 
	\begin{align*}
	\sum_{j=0}^n \binom {2j}{j} 4^{n-j} O_j H_{n-j}  &= 2^{2n-1} (n+1) \Big ( (1-H_{n+1})^2 + 1-H_{n+1}^{(2)}\Big) \\
	&\quad - \sum_{j=1}^n C_{n-j} 4^{j-1} j \Big ( (1-H_{j})^2 + 1-H_{j}^{(2)}\Big ).
	\end{align*}
\end{corollary}
\begin{proof}
	Work with the relation
	\begin{equation*}
	\sum_{n=0}^\infty 4^n H_n x^n \cdot \sum_{n=0}^\infty \binom {2n}{n} O_n x^n
	= \frac{1}{2} \sqrt{1-4x} \cdot\Big ( - \frac{\ln (1-4x)}{1-4x} \Big )^2
	\end{equation*}
	and use the convolution identity \cite{Spiess} $$\sum\limits_{j=0}^n H_j H_{n-j} = (n+1) \left((1-H_{n+1})^2 + 1-H_{n+1}^{(2)}\right).
	$$
\end{proof}

\section{Two additional combinatorial identities derived from the integral}

The goal of this section is to derive two additional combinatorial identities using the integral in Lemma \ref{fund_lem}.
These identities are stated in Theorem \ref{main_thm5}. To prove this theorem, two lemmas will be needed.
\begin{lemma}\label{temp_lem1}
	For integers $n\geq 0$,
	\begin{equation*}
	\frac{d^n}{da^n} \frac{\ln a}{a} = (-1)^n n! \,\frac{ \ln a - H_n}{a^{n+1}} .
	\end{equation*}
\end{lemma}
\begin{proof}
	This follows easily upon applying Leibniz rule for derivatives in conjunction with
	\begin{equation*}
	\frac{d^n}{da^n} \ln a = (-1)^{n-1} (n-1)!\, a^{-n}\quad \text{and}\quad \frac{d^n}{da^n} a^{-1} = (-1)^{n} n! \,a^{-(n+1)}.
	\end{equation*}
\end{proof}
\begin{lemma}\label{temp_lem2}
	For integers $n\geq 0$, 
	\begin{equation*}
	\frac{d^n}{da^n} \frac{1}{x^2+a^2}
	= \sum_{k=0}^{\lfloor (n+1)/2 \rfloor} (-1)^{n-k} (2a)^{n-2k} \frac{(n-k)!}{k!} \frac{(n+1-2k)_{2k}}{(x^2+a^2)^{n+1-k}},
	\end{equation*}
	where $(\lambda)_n$ is the Pochhammer symbol.
\end{lemma}
\begin{proof}
	The first two expressions
	\begin{equation*}\frac{d}{da} \frac{1}{x^2+a^2} = \frac{-2a}{(x^2+a^2)^2}\quad\text{and}\quad \frac{d^2}{da^2} \frac{1}{x^2+a^2} = \frac{6a^2-2x^2}{(x^2+a^2)^3}
	\end{equation*}	
	indicate that we can assume that the $n$th derivative can be written in the form
	\begin{equation*}
	\frac{d^n}{da^n} \frac{1}{x^2+a^2} = \frac{Q_n(a,x)}{(x^2+a^2)^{n+1}},
	\end{equation*}
	where $Q_n(a,x)$ is a polynomial in $a$ of degree $n$. Then
	\begin{equation*}
	\frac{d^{n+1}}{da^{n+1}} \frac{1}{x^2+a^2} = \frac{Q_{n+1}(a,x)}{(x^2+a^2)^{n+2}}
	\end{equation*}
	with $$Q_{n+1}(a,x) = \frac{d}{da} Q_n(a,x) (a^2+x^2) - 2(n+1)a Q_n(a,x)$$ and $Q_0(a,x)=1$.
	One checks with little effort that the polynomial
	\begin{equation*}
	Q_n(a,x) = \sum_{k=0}^n (-1)^{n-k} (2a)^{n-2k}  \frac{(n-k)!}{k!} (n+1-2k)_{2k} (a^2+x^2)^k
	\end{equation*}
	satisfies these conditions. Finally, as
	$(n+1-2k)_{2k} = 0$ for $\lceil (n+1)/2 \rceil \leq k \leq n,$
	the result follows.
\end{proof}
\begin{theorem}\label{main_thm5}
	Let $(\lambda)_n$ be the Pochhammer symbol. Then for integers $n\geq 0$, we have
	\begin{align*}
	\sum_{k=0}^{\lfloor (n+1)/2 \rfloor} (-1)^k \frac{\binom {2(n-k)}{n-k}}{\binom {n}{k}} \frac{(n+1-2k)_{2k}}{(k!)^2} &= 2^n,\\
	\sum_{k=0}^{\lfloor (n+1)/2 \rfloor} (-1)^k \frac{\binom {2(n-k)}{n-k}}{\binom {n}{k}} \frac{(n+1-2k)_{2k}}{(k!)^2} O_{n-k} &= 2^n H_n.
	\end{align*}
\end{theorem}
\begin{proof}
	Consider \eqref{lem_eq} for $n=0$, i.e.,
	$	\int\limits_0^\infty \frac{\ln x}{x^2+a^2}\, dx = \frac{\pi}{2}\frac{\ln a}{a}$.
	Now, differentiating both sides of the equation $n$ times with respect to $a$ while using Lemmas \ref{temp_lem1} and  \ref{temp_lem2} produces
	\begin{align*}
	&\sum_{k=0}^{\lfloor (n+1)/2 \rfloor} (-1)^{n-k} (2a)^{n-2k} \frac{(n-k)!}{k!} (n+1-2k)_{2k}
	\int_0^\infty \frac{\ln x}{(x^2+a^2)^{n+1-k}}\,dx \\
	&\qquad= \sum_{k=0}^{\lfloor (n+1)/2 \rfloor} (-1)^{n-k} (2a)^{n-2k} \frac{(n-k)!}{k!} (n+1-2k)_{2k}
	\binom {2(n-k)}{n-k} \frac{\pi (\ln a-O_{n-k})}{(2a)^{2(n-k)+1}} \\
	&\qquad= (-1)^n \frac{\pi n!}{2a^{n+1}} \left ( \ln a - H_n\right ).
	\end{align*}
	This gives after some obvious simplifications
	\begin{equation*}
	\sum_{k=0}^{\lfloor (n+1)/2 \rfloor} (-1)^k \frac{\binom {2(n-k)}{n-k}}{\binom {n}{k}} \frac{(n+1-2k)_{2k}}{(k!)^2}
	\left (\ln a - O_{n-k}\right ) = 2^n \left (\ln a - H_n\right ),
	\end{equation*}
	and this completes the proof.
\end{proof}

\section{Some Fibonacci series identities}

Next we present selected series evaluations involving products of Fibonacci (Lucas) numbers and odd harmonic numbers.

Generating function \eqref{chen_id} when evaluated at $x=s\alpha$ and $x=s\beta$ with $s<\frac{1}{4\alpha}\approx 0.155$ yields:
\begin{align*}
\sum_{n=1}^\infty \binom {2n}{n} s^n O_n L_{n+t} &= -\frac12\left( \frac{\alpha^t\ln(1-4s\alpha)  }{\sqrt{1-4s\alpha}}
+ \frac{\beta^t\ln(1-4s\beta)    }{\sqrt{1-4s\beta}}\right),\\
\sum_{n=1}^\infty \binom {2n}{n} s^n O_n F_{n+t} &= -\frac{\sqrt5}{10}\left( \frac{\alpha^t\ln(1-4s\alpha) }{\sqrt{1-4s\alpha}}
- \frac{\beta^t \ln(1-4s\beta)}{\sqrt{1-4s\beta}}\right),
\end{align*}
and, more generally,
\begin{align*}
\sum_{n=1}^\infty \binom {2n}{n}  s^n O_n G_{n+t}= -\frac{\sqrt5}{10}\left( \frac{(b-a\beta)\alpha^{t}}{\sqrt{1-4s\alpha}}\ln(1-4s\alpha)
-\frac{(b-a\alpha)\beta^{t}}{\sqrt{1-4s\beta}}\ln(1-4s\beta)\right)\!.
\end{align*}

This yields the next results, each of which presented as a separate theorem.
\begin{theorem}
	If $t$ is any integer, then
	\begin{align}
	\sum_{n=1}^\infty \binom {2n}{n} \frac{O_n}{8^n} L_{n+t} &= \sqrt{2}  L_{t+1}\ln\alpha + \frac{\sqrt{10}\ln2}{2}F_{t+1} ,\nonumber\\
	\label{12}
	\sum_{n=1}^\infty \binom {2n}{n} \frac{O_n}{8^n} F_{n+t} &= \sqrt{2}  F_{t+1}\ln\alpha + \frac{\ln2}{\sqrt{10}} L_{t+1},
	\end{align}
	and, more generally,
	\begin{equation*}
	\sum_{n=1}^\infty \binom {2n}{n} \frac{O_n}{8^n} G_{n+t} = \sqrt{2}  G_{t+1} \ln\alpha+ \frac{\ln 2}{\sqrt{10}} (G_{t+2} + G_t).
	\end{equation*}
\end{theorem}

Note that \eqref{12} is a generalization of Chen \cite[Identity (7)]{Chen}. \begin{theorem}
	If $t$ is any integer, then
	\begin{align*}
	\sum_{n=1}^\infty (-1)^{n-1}\binom {2n}{n} \frac{O_n}{8^n} L_{n+t} &=\frac{\sqrt{2}}{2\sqrt{\alpha+2}}\left(\big(\alpha^{t+1}-L_{t-1}\big)  \ln\Big(\frac{\sqrt5}2\Big)+\big(\alpha^{t-2}+L_{t-1}\big)\ln\alpha\right)\!,\\
	\sum_{n=1}^\infty (-1)^{n-1}\binom {2n}{n} \frac{O_n}{8^n} F_{n+t} &=\frac{\sqrt{10}}{10\sqrt{\alpha+2}}\left(\big(\alpha^{t-2}+L_{t-1}\big)  \ln\Big(\frac{\sqrt5}2\Big)+\big(\alpha^{t+1}- L_{t-1}\big)\ln\alpha\right)\!,
	\end{align*}
	and, more generally,
	\begin{align*}
	\sum_{n=1}^\infty (-1)^{n-1}\binom {2n}{n} \frac{O_n}{8^n} G_{n+t} &=\frac{\sqrt{10}}{10\sqrt{\alpha+2}}\Big(\big(\alpha^{t-3}(b\alpha+a)+G_{t}+G_{t-2}\big)  \ln\Big(\frac{\sqrt5}2\Big)\\
	&\quad +\big(\alpha^{t}(b\alpha+a)- G_{t}-G_{t-2}\big)\ln\alpha\Big)\!.
	\end{align*}
\end{theorem}
\begin{theorem}
	If $t$ is any integer, then
	\begin{align*}
	\sum_{n=1}^\infty \binom {2n}{n} \frac{O_n}{16^n} L_{n+t} &=\frac{1}{\sqrt{\alpha+2}}\left(\big(\alpha^{t+2}-L_{t}\big)  \ln\alpha- \big(\alpha^{t-1}+L_{t}\big)\ln\Big(\frac{\sqrt5}{4}\Big)\right),\\
	\sum_{n=1}^\infty \binom {2n}{n} \frac{O_n}{16^n} F_{n+t} &= \frac{\sqrt{5}}{5\sqrt{\alpha+2}}\left(\big(\alpha^{t-1}+L_{t}\big)  \ln\alpha-\big(\alpha^{t+2}- L_{t}\big)\ln\Big(\frac{\sqrt5}{4}\Big)\right)\!,
	\end{align*}
	and
	\begin{align*}
	\sum_{n=1}^\infty \binom {2n}{n} \frac{O_n}{16^n} G_{n+t} & =\frac{\sqrt{5}}{5\sqrt{\alpha+2}}\Bigg(\big(\alpha^{t-2}(b\alpha+a)+G_{t+1}+G_{t-1}\big)  \ln\alpha\\
	&\quad- \big(\alpha^{t+1}(b\alpha+a)- G_{t+1}-G_{t-1}\big)\ln\Big(\frac{\sqrt5}{4}\Big)\Bigg).
	\end{align*}
\end{theorem}
\begin{theorem}
	If $t$ is any integer, then 
	\begin{align*}
	\sum_{n=1}^\infty(-1)^{n-1}\binom {2n}{n} \frac{O_n}{16^n} L_{2n+t} &= \frac{\alpha^t}{\sqrt{\alpha+5}}\ln\Big(\frac{\alpha+5}{4}\Big)+\frac{\beta^t}{\sqrt{\beta+5}}\ln\Big(\frac{\beta+5}{4}\Big),\\
	\sum_{n=1}^\infty(-1)^{n-1}\binom {2n}{n} \frac{O_n}{16^n} F_{2n+t} &= \frac{1}{\sqrt5}\left(\frac{\alpha^t}{\sqrt{\alpha+5}}\ln\Big(\frac{\alpha+5}{4}\Big)-\frac{\beta^t}{\sqrt{\beta+5}}\ln\Big(\frac{\beta+5}{4}\Big)\right),
	\end{align*}
	and
	\begin{align*}
	\sum_{n=1}^\infty(-1)^{n-1}\binom {2n}{n} \frac{O_n}{16^n} G_{2n+t} = \frac{1}{\sqrt5}\left(\frac{(b-a\beta)\alpha^t}{\sqrt{\alpha+5}}\ln\Big(\frac{\alpha+5}{4}\Big)-\frac{(b-a\alpha)\beta^t}{\sqrt{\beta+5}}\ln\Big(\frac{\beta+5}{4}\Big)\right)\!.
	\end{align*}
\end{theorem}
\begin{theorem}
	If $t$ is any integer, then
	\begin{align*}
	\sum_{n=1}^\infty \binom {2n}{n} \frac{O_n L_{n+t}}{12^n} &=
	\frac{\sqrt{15}\sqrt{\alpha+2}}{10}\left((\alpha^t+\beta^{t+1})\ln\alpha-(\alpha^t-\beta^{t+1})\ln\Big(\frac{\sqrt5}{3}\Big)\right),\\
	\sum_{n=1}^\infty \binom {2n}{n}  \frac{O_n F_{n+t}}{12^n} &=
	\frac{\sqrt{3}\sqrt{\alpha+2}}{10}\left((\alpha^t-\beta^{t+1})\ln\alpha-(\alpha^t+\beta^{t+1})\ln\Big(\frac{\sqrt5}{3}\Big)\right).
	\end{align*}
	and
	\begin{align*}
	\sum_{n=1}^\infty \binom {2n}{n}  \frac{O_n G_{n+t}}{12^n} &=	\frac{\sqrt{3}\sqrt{\alpha+2}}{10}\Bigg(
	\big(a(\alpha^{t-1}-\beta^{t})+b(\alpha^{t}-\beta^{t+1})\big)\ln\alpha\\
	&\quad-\big(a(\alpha^{t-1}+\beta^{t})+b(\alpha^{t}+\beta^{t+1})\big)
	\ln\Big(\frac{\sqrt5}{3}\Big)\Bigg).
	\end{align*}
\end{theorem}
\begin{theorem}
	If $t$ is any integer, then
	\begin{align*}\label{61}
	\sum_{n=1}^\infty \binom {2n}{n} \frac{O_n L_{2n+t}}{12^n} &= \frac{\sqrt{3}}{2}\big( 2L_{t+1}\ln\alpha + \sqrt{5} F_{t+1}\ln 3\big),\\
	\sum_{n=1}^\infty \binom {2n}{n} \frac{O_n F_{2n+t}}{12^n} &= \frac{\sqrt{15}}{10}\big(2\sqrt{5} F_{t+1}\ln\alpha + L_{t+1}\ln 3\big),\\
	\sum_{n=1}^\infty \binom {2n}{n} \frac{O_n G_{2n+t}}{12^n} &= \frac{\sqrt{15}}{10}\left(2\sqrt{5}\, G_{t+1}\ln\alpha + (G_{t+2}+G_t)\ln 3\right)\!.
	\end{align*}
\end{theorem}
\begin{lemma}\label{lem.rsoebsn}
	If $r$ is an integer, then
	\begin{align*}
	\sqrt {\alpha ^r }  \pm \sqrt {\beta ^r }  &= \sqrt {L_r  \pm 2} ,\qquad\text{$r$ even},\\
	\sqrt {\alpha ^r }  \pm \sqrt { - \beta ^r }  &= \sqrt {\sqrt 5\,F_r   \pm 2},\qquad\text{$r$ odd} .
	\end{align*}
\end{lemma}
\begin{proof}
	The veracity of each identity is readily established by squaring both sides.
\end{proof}
\begin{theorem}
	If $r$ is an even integer and $t$ is an integer, then
	\begin{align*}
	\sum_{n=0}^{\infty} \binom {2n}{n} \frac{O_nL_{rn+t}}{4^nL_r^n}
	&=\frac{\sqrt{L_r}}{2}\left( \sqrt{L_rL_{r+2t}+(-1)^t2}\,\ln L_r + {r} \sqrt{L_{r+2t}-(-1)^t2}\,\ln \alpha\right)\!,\\
	\sum_{n=0}^{\infty} \binom{2n}{n} \frac{O_nF_{rn+t}}{4^nL_r^n}
	&=\frac{\sqrt{5L_r}}{10}\left( \sqrt{L_{r+2t}-(-1)^t2} \,\ln L_r + {r} \sqrt{L_{r+2t}+(-1)^t2} \,\ln \alpha\right)\!.
	\end{align*}
\end{theorem}
\begin{proof}
	Set $x=\alpha^r/(4L_r)$ and $x=\beta^r/(4L_r)$, in turn, in \eqref{chen_id}. Multiply through the resulting identities by
	$\beta^t$ and $\alpha^t$, respectively. Combine according to the Binet formulas \eqref{bine} and make use of Lemma \ref{lem.rsoebsn}.
\end{proof}
In particular,
\begin{align*}
\sum_{n = 0}^\infty \binom{2n}{n} \frac{O_n L_{2n}}{12^n} 
&= \frac{{\sqrt 3 }}{2}\big( {\sqrt 5 \ln 3 + 2\ln\alpha } \big),\\
\sum_{n = 0}^\infty \binom{2n}{n} \frac{O_n F_{2n}}{12^n} 
&= \frac{\sqrt3}{10} \big( \sqrt5\ln 3 + 10 \ln\alpha\big),\\
\sum_{n = 1}^\infty \binom{2n}{n} \frac{O_n L_{2(n-1)}}{12^n} 
&= \frac{{\sqrt 3 }}{2}\big( {\sqrt 5 \ln 3 - 2\ln\alpha } \big),\\
\sum_{n = 1}^\infty \binom{2n}{n} \frac{O_n F_{2(n-1)}}{12^n} 
&=- \frac{\sqrt{3}}{10} \big(\sqrt5\ln 3 - 10\ln\alpha\big).
\end{align*}

\section{Several other generating functions and series}

\begin{theorem}
	If $p$ is a non-negative integer and $t$ is any integer, then
	\begin{equation}\label{eq.c5ncid9}
	\begin{split}
	\sum_{n = 0}^\infty & \binom{n + p}p\frac{{H_n G_{n + t} }}{{2^n }} = \big( {H_p  + \ln 2} \big)2^{p + 1}G_{t + 2p + 2}  \\
	&\quad+ \frac{{2^{p + 2} }}{{\sqrt 5 }}\left( {G_{t + 2p + 3}  + G_{t + 2p + 1} } \right)\ln\alpha - 2^{p + 1} \sum_{k = 1}^p {\frac{{G_{t + 2p - 2k + 2} }}{{2^k k}}}.
	\end{split}
	\end{equation}
	If $p$ is an even integer, then
	\begin{equation}\label{eq.qkd3rtj}
	\begin{split}
	&\sum_{n = 0}^\infty ( - 1)^n \binom{n + p}p\frac{H_n G_{n + t}}{2^n} \\
	&\qquad= \Big(\frac{2}{\sqrt 5} \Big)^{p +1}\left(\frac{H_p  - \ln\big( \sqrt 5/2 \big)}{{\sqrt5}}\left( G_{t - p}  + G_{t - p - 2}  \right) - G_{t - p - 1} \ln\alpha\right) \\
	&\qquad\quad - \sum_{k = 1}^{p/2} \frac{\big(\frac{4}{5}\big)^{p/2-k}}{5k}\left(G_{t + 2k - p}  +\frac{6k-1}{2k-1} G_{t + 2k - p - 2}\right)\!,
	\end{split}
	\end{equation}
	while if $p$ is an odd integer, then
	\begin{equation}\label{eq.qwfkq8r}
	\begin{split}
	&\sum_{n = 0}^\infty ( - 1)^n \binom{n + p}{p}\frac{H_n G_{n + t}}{2^n}  \\
	&\quad = \Big(\frac{2}{\sqrt 5} \Big)^{p+1}\Big( \big( H_p - \ln \big(\sqrt 5/2\big)\big) G_{t - p - 1} - \frac{\ln\alpha}{\sqrt 5} \big(G_{t - p} + G_{t - p - 2} \big)\Big)\\
	&\qquad - \frac{8}{25}\sum_{k = 1}^{\frac{p - 1}{2}} \frac{\big(\frac45\big)^{\frac{p-1}{2}-k}}{2k-1}  \left(\frac{14k-5}{4k} G_{t + 2k - p - 1} + G_{t + 2k - p - 3}\right)-\frac{2(G_{t}+G_{t-2})}{5p}.
	\end{split}
	\end{equation}
\end{theorem}
\begin{proof}
	To prove \eqref{eq.c5ncid9}, set $x=\alpha/2$ and $x=\beta/2$, in turn, in \eqref{eq.rqnwwad} and combine according to the Binet formula \eqref{binet-gibo}. Using $x=-\alpha/2$ and $x=-\beta/2$, in turn, in \eqref{eq.rqnwwad} and combining in accordance with \eqref{binet-gibo} and the parity of $p$ produces \eqref{eq.qkd3rtj} and \eqref{eq.qwfkq8r}.
\end{proof}

Chen \cite[Theorem 6]{Chen} has shown the following generating function:
\begin{equation*}
\sum_{n=0}^{\infty}\binom{2n}{n} H_{2n}x^n = \frac{1}{\sqrt{1-4x}}\ln \Big(\frac{1+\sqrt{1-4x}}{2(1-4x)}\Big),\quad  |x|<\frac14.
\end{equation*}
The companion result involving odd-indexed harmonic numbers is the next result.
\begin{theorem}
	For all $|x|<1/4$, we have
	\begin{align}\label{thm_gf1}
	\sum_{n=0}^{\infty}	\binom{2n}{n} H_{2n+1} x^n& = \frac{1}{\sqrt{1-4x}}\ln \Big(\frac{1+\sqrt{1-4x}}{2(1-4x)}\Big)
	+\frac{\arcsin(2\sqrt{x})}{2\sqrt{x}},\\
	\sum_{n=0}^{\infty} (-1)^n \binom{2n}{n}H_{2n+1} x^n &=  \frac{1}{\sqrt{1+4x}}\ln \Big(\frac{1+\sqrt{1+4x}}{2(1+4x)}\Big) + \frac{\ln\big(2\sqrt{x}+\sqrt{1+4x}\big)}{2\sqrt{x}}.\nonumber
	\end{align}
\end{theorem}
\begin{proof}
	Using \eqref{eq.lreuv6s}, we have
	\begin{equation*}
	\sum_{n = 0}^\infty \binom{2n}{n} H_{2n + 1} x^n = \frac{1}{2}\sum_{n = 0}^\infty \binom {2n}{n} H_n x^n
	+ \sum_{n = 0}^\infty \binom{2n}{n} O_{n + 1} x^n.
	\end{equation*}
	Substituting \eqref{eq.royfy2s} and \eqref{main_id1}, and noting that
	$O_{n + 1} = O_n + \frac{1}{2n + 1}$, as well as
	\begin{align*}
	\sum_{n = 0}^\infty \binom{2n}{n} \frac{x^n}{2n + 1} &= \frac{{\arcsin(2\sqrt x )}}{{2\sqrt x }},\\
	\sum_{n = 0}^\infty(- 1)^{n} \binom{2n}{n}  \frac{x^n}{2n + 1} &= \frac{{\ln \big(2\sqrt x  + \sqrt {1 + 4x} \big)}}{{2\sqrt x }},
	\end{align*}
	the results follow.
\end{proof}
\begin{corollary}
	For all $|x|<1/4$, we have
	\begin{align*}
	\sum_{n=0}^{\infty}	C_n H_{2n+1} x^n &= \frac{1}{2x}\Big(\ln 2 -\big(1+\sqrt{1-4x}\big)\ln \big(1+\sqrt{1-4x}\big)  \nonumber \\
	& \quad+ \sqrt{1-4x}\ln\big(2(1-4x)\big) + 2\sqrt{x} \arcsin(2\sqrt{x}) \Big ).
	\end{align*}
\end{corollary}
\begin{proof}
	Integrate \eqref{thm_gf1} with respect to $x$ while using   \cite[Corollary 7]{Chen} and the elementary integral
	$	\int \frac{\arcsin(2\sqrt{x})}{2\sqrt{x}}\,dx = \frac{1}{2}\sqrt{1-4x} + \sqrt{x} \arcsin(2\sqrt{x})$.
\end{proof}

Chen \cite[Corollary 9]{Chen} has also shown that for all $|x|<1/4$, 
\begin{equation}\label{chen_cor}
\sum_{n=0}^{\infty} C_{n} O_{n} x^n = \frac{1}{2x}\Big (1 - \sqrt{1-4x} + \sqrt{1-4x}\, \ln\sqrt{1-4x} \,\Big).
\end{equation}
His result can be generalized as follows.
\begin{theorem}
	For all integers $m\geq 0$ and all $x$ with $|x|<1/4$, we have
	\begin{align*}
	\sum_{n=0}^\infty \binom{2n}{n} \frac{O_{n}}{n+m+1} x^n &= \frac{\sqrt{1-4x}}{2x}\big (\ln\sqrt{1-4x}-1\big) + \frac{0^m}{2x^{m+1}}\\
	& \quad + \frac{m}{4^m x^{m+1}}\sum_{j=0}^{m-1} (-1)^{j}\binom {m-1}{j}  A_j(x)
	\end{align*}
	with
	$A_j(x) = \frac{(1-4x)^{j+3/2}}{2j+3}\big(\ln\sqrt{1-4x} - \frac{2j+4}{2j+3}\big) + \frac{2j+4}{(2j+3)^2}$.
\end{theorem}
\begin{proof}
	We start with
	\begin{equation*}
	\sum\limits_{n=0}^\infty \binom {2n}{n} O_n x^{n+m} = - x^m \frac{\ln(1-4x)}{2\sqrt{1-4x}},\quad m\geq 0.
	\end{equation*}
	Integration by parts produces
	\begin{align*}
	\sum_{n=0}^\infty \binom{2n}{n} \frac{O_{n}y^{n+m+1}}{n+m+1}  = \frac{y^m\sqrt{1-4y}}{2}\big (\ln\sqrt{1-4y}-1\big) + \frac{0^m}{2} - \frac{m}{2}(I_1(y) - I_2(y)),
	\end{align*}
	with
	\begin{equation*}
	I_1(y) = \int\limits_0^y x^{m-1}\,\sqrt{1-4x}\ln\sqrt{1-4x}\,dx,\quad
	I_2(y) = \int\limits_0^y x^{m-1}\,\sqrt{1-4x}\,dx.
	\end{equation*}
	Both integrals are elementary. We have ($z=1-4x$)
	\begin{align*}
	I_2(y) &= -\Big (\frac{1}{4}\Big )^{m} \sum_{j=0}^{m-1}(-1)^j \binom {m-1}{j}  \int_1^{1-4y} z^j \sqrt{z}\,dz \\
	&= 2^{1-2m} \sum_{j=0}^{m-1}(-1)^{j} \binom {m-1}{j}  \frac{1- (1-4y)^{j+3/2}}{2j+3}
	\end{align*}
	and
	\begin{align*}
	I_1(y) &= 2^{-1-2m} \sum_{j=0}^{m-1} (-1)^{j+1}\binom {m-1}{j}  \int_1^{1-4y} z^j \sqrt{z} \ln z \,dz \\
	&= 2^{1-2m} \sum_{j=0}^{m-1}(-1)^{j+1} \binom {m-1}{j}  \Big ( \frac{(1-4y)^{j+3/2}}{2j+3}  \ln\sqrt{1-4y} \\
	& \quad - \frac{(1-4y)^{j+3/2} }{(2j+3)^2} + \frac{1}{(2j+3)^2} \Big ),
	\end{align*}
	as
	\begin{equation*}
	\int z^n \sqrt{z} \ln z \,dz = \frac{2z^{n+3/2}}{2n+3}  \ln z - \frac{4z^{n+3/2}}{(2n+3)^2} .
	\end{equation*}
	The result follows upon simplifying and changing the variable $y$ to $x$.
\end{proof}

When $m=0$ we get \eqref{chen_cor}. When $m=1$ then we get the following generating function.
\begin{corollary}
	For all $x$ with $|x|<1/4$, we have
	\begin{align*}
	\sum_{n=0}^\infty \binom{2n}{n} \frac{O_{n}}{n+2} x^n = \frac{\sqrt{1-4x}}{2x} \big (\ln\sqrt{1-4x}-1 \big) + \frac{(1-4x)^{3/2}}{36x^2}\big (3\ln\sqrt{1-4x} - {4}\big) + \frac{1}{9x^2}.
	\end{align*}
\end{corollary}
In particular,
\begin{align*}
\sum_{n=0}^\infty \binom{2n}{n} \frac{O_{n}}{(n+2)8^n} &= \frac{1}{9} \Big ( 64 - 34\sqrt{2} - 15\sqrt{2}\ln 2 \Big ),\\
\sum_{n=0}^\infty (-1)^n \binom{2n}{n} \frac{O_{n}}{(n+2)8^n}  &= \frac{1}{9} \Big(64-30\sqrt{6}+9\sqrt{6}\ln\Big(\frac{3}{2}\Big)\Big ),
\end{align*}
and
\begin{align*}
\sum_{n=0}^\infty (-1)^n \binom{2n}{n} \frac{O_{n}L_{n+t}}{(n+2)8^n}  &= \frac{64}{9}L_{t-2}-\sqrt{10}\Big(F_{n-2}+\frac23F_{t-5}\Big)\ln2
\\
&\quad-2\sqrt{2}\Big(L_{n-2}+\frac23L_{t-5}\Big)\ln\alpha-2\sqrt{10}\Big(F_{n-2}+\frac89F_{t-5}\Big),\\
\sum_{n=0}^\infty (-1)^n \binom{2n}{n} \frac{O_{n}F_{n+t}}{(n+2)8^n}  &= \frac{64}{9}F_{t-2} -\frac{\sqrt{10}}{5}\Big(L_{n-2}+\frac23L_{t-5}\Big)\ln2
\\
&\quad-2\sqrt{2}\Big(F_{n-2}+\frac23F_{t-5}\Big)\ln\alpha-\frac{2\sqrt{10}}{5}\Big(L_{n-2}+\frac89L_{t-5}\Big).
\end{align*}
\begin{theorem}
	For all $|x|<1/4$, we have
	\begin{equation}\label{O2n_id}
	\sum_{n=0}^\infty \binom {4n}{2n} O_{2n} x^{2n} = - \frac{\sqrt{1+4x}\ln(1-4x)+\sqrt{1-4x}\ln(1+4x)}{4\sqrt{1-16x^2}}
	\end{equation}
	and
	\begin{equation}
	\label{O2n_id2}
	\sum_{n=0}^\infty \binom {4n+2}{2n+1} O_{2n+1} x^{2n+1} = -  \frac{\sqrt{1+4x}\ln(1-4x)-\sqrt{1-4x}\ln(1+4x)}{4\sqrt{1-16x^2}}.
	\end{equation}
\end{theorem}
\begin{proof} 	Calculate
	$\frac{1}{2} (f(x) + f(-x))$ and $\frac{1}{2} (f(x) - f(-x))$
	with $f(x)$ being the generating function \eqref{main_id1}.
\end{proof}
\begin{theorem}\label{th17}
	If $t$ is any nonnegative integer, then
	\begin{align*}
	&\sum_{n=1}^\infty \binom {4n}{2n} \frac{F_{2n+t}}{64^n} O_{2n} \\
	&= \frac{\sqrt{2}}{20} \begin{cases}
	\begin{aligned}
	& \left (10F_{t+1} -  \sqrt{\sqrt{5}}\, \sqrt{\sqrt{5}F_{2t-1}+2} \right)\ln\alpha \\
	& +\left (\sqrt5 L_{t+1}+\sqrt{\sqrt5}\, \sqrt{\sqrt{5}F_{2t-1}-2} \right)\ln2 
	-\frac{\sqrt{\sqrt5}}{2}\sqrt{\sqrt{5}F_{2t-1}-2}\,\ln5;
	\end{aligned} & \text{$t$ even}; \\[22pt]
	\begin{aligned}
	& \left (10F_{t+1} -  \sqrt{\sqrt{5}}\, \sqrt{\sqrt{5}F_{2t-1}-2} \right)\ln\alpha \\
	& +\left (\sqrt5 L_{t+1}+\sqrt{\sqrt5}\, \sqrt{\sqrt{5}F_{2t-1}+2} \right)\ln2 
	-\frac{\sqrt{\sqrt5}}{2}\sqrt{\sqrt{5}F_{2t-1}+2}\,\ln5.
	\end{aligned} & \text{$t$ odd}; 
	\end{cases}
	\end{align*}
	\begin{align*}
	&\sum_{n=1}^\infty \binom {4n}{2n} \frac{L_{2n+t}}{64^n} O_{2n} \\
	&=\frac{\sqrt{2}}{4} \begin{cases}
	\begin{aligned}
	& \left (2L_{t+1} -  \frac{1}{\sqrt{\sqrt{5}}}\, \sqrt{\sqrt{5}F_{2t-1}-2} \right)\ln\alpha \\
	& +\left (\sqrt5 F_{t+1}+\frac{1}{\sqrt{\sqrt5}}\, \sqrt{\sqrt{5}F_{2t-1}+2} \right)\ln2 
	-\frac{1}{2\sqrt{\sqrt5}}\sqrt{\sqrt{5}F_{2t-1}+2}\,\ln5;
	\end{aligned} & \text{$t$ even}; \\[22pt]
	\begin{aligned}
	& \left (2L_{t+1} -  \frac{1}{\sqrt{\sqrt{5}}}\, \sqrt{\sqrt{5}F_{2t-1}+2} \right)\ln\alpha \\
	& +\left (\sqrt5 F_{t+1}+\frac{1}{\sqrt{\sqrt5}}\, \sqrt{\sqrt{5}F_{2t-1}-2} \right)\ln2 
	-\frac{1}{2\sqrt{\sqrt5}}\sqrt{\sqrt{5}F_{2t-1}-2}\,\ln5. 
	\end{aligned} & \text{$t$ odd}. 
	\end{cases}
	\end{align*}
\end{theorem}
\begin{proof}
	Insert $x=\alpha/8$ and $x=\beta/8$ in \eqref{O2n_id}, respectively, and multiply through by $\alpha^t$ (resp. $\beta^t$). This yields
	\begin{equation*}
	\sum_{n=1}^\infty \binom {4n}{2n} O_{2n} \frac{\alpha^{2n+t}}{64^n} =
	- \frac{\sqrt2}{4\sqrt{\sqrt5}} \left ( \sqrt{\sqrt{5}}\, \alpha^{t+1} \ln \big(\beta^2/2\big) + \sqrt{\alpha}\,\alpha^{t-1}\ln\big(\sqrt{5}\alpha/2\big)\right )
	\end{equation*}
	and
	\begin{equation*}
	\sum_{n=1}^\infty \binom {4n}{2n} O_{2n} \frac{\beta^{2n+t}}{64^n} =
	\frac{\sqrt2}{4\sqrt{\sqrt{5}\alpha}} \left (\sqrt{\sqrt5} \beta^{t+1} \ln \big(\alpha^2/2\big) + \beta^{t-2}\ln\big(\sqrt{5}/(2\alpha)\big)\right )\!.
	\end{equation*}
	When combining according to the Binet formulas \eqref{bine} use the algebraic relations  
	\begin{equation*}
	\sqrt{\alpha} \alpha^{t} \pm \sqrt{-\beta}\beta^{t} =
	\begin{cases}
	\sqrt{\sqrt{5}F_{2t+1} \pm  2}, & \text{if $t$ is even}; \\
	\sqrt{\sqrt{5}F_{2t+1} \mp 2}, & \text{if $t$ is odd,} 
	\end{cases}
	\end{equation*}
	applicable to nonnegative integer $t$.
\end{proof}
In particular, Theorem \ref{th17} yields
\begin{align*}
\sum_{n=1}^\infty &\binom {4n}{2n} \frac{F_{2n+1}}{64^n} O_{2n} \\
&=\frac{\sqrt2}{20}
\left( \Big(10-\sqrt{5-2\sqrt{5}}\Big) \ln\alpha+\Big(3\sqrt{5}+\sqrt{5+2\sqrt{5}}\Big) \ln2 -\frac12\sqrt{5+2\sqrt{5}}\,\ln5\right)
\end{align*}
and
\begin{align*}
\sum_{n=1}^\infty &\binom {4n}{2n} \frac{L_{2n+1}}{64^n} O_{2n} \\
&=\frac{\sqrt{10}}{20}
\left( \Big(6\sqrt5-\sqrt{5+2\sqrt{5}}\Big)\ln\alpha+\Big(5+\sqrt{5-2\sqrt{5}}\Big) \ln2 -\frac12\sqrt{5-2\sqrt{5}}\,\ln5\right)\!.
\end{align*}

A similar result can be derived from \eqref{O2n_id2} which is left to the interested reader.

\section{A family of related integrals}

From Lemma \ref{fund_lem} we also get the next result where odd harmonic numbers are involved.
\begin{theorem}
	For each integers $m\ge 1$ and $n\geq 0$, we have
	\begin{align*}
	\int_0^\infty & \frac{\sum\limits_{j=0}^{n+1} \binom {n+1}{j}  L_{4m(n+1-j)}x^{2j}}{(x^4+L_{4m}x^2+1)^{n+1}}\,\ln x\,dx\\
	&\qquad\qquad= - \binom {2n}{n} \frac{\pi}{2^{2n+1}} \Big ( 2m \sqrt{5}  F_{2m(2n+1)} \ln \alpha+ O_n L_{2m(2n+1)} \Big )
	\end{align*}
and
	\begin{align*}
	\int_0^\infty  &\frac{\sum\limits_{j=0}^{n+1} \binom {n+1}{j}  F_{4m(n+1-j)}x^{2j}}{(x^4+L_{4m}x^2+1)^{n+1}}\ln x\,dx\\
	&\qquad\qquad= - \binom {2n}{n} \frac{\pi}{2^{2n+1}} \Big ( \frac{2m}{\sqrt{5}}  L_{2m(2n+1)}\ln \alpha + O_n F_{2m(2n+1)} \Big ).
	\end{align*}
\end{theorem}
\begin{proof}
	Set $a=\alpha^{2m}$ and $a=\beta^{2m}$ in \eqref{lem_eq} and combine according to the Binet formulas \eqref{bine}. When simplifying
	make use of $\beta^{2m}=\alpha^{-2m}$.
\end{proof}

In particular, for $m\geq1$,
\begin{align*}
\int_0^\infty \frac{2x^2 + L_{4m}}{x^4+L_{4m}x^2+1}\ln x\,dx &= - m \sqrt{5} \pi F_{2m}\ln\alpha,\\
\int_0^\infty \frac{\ln x}{x^4+L_{4m}x^2+1}\,dx &= - \frac{m \pi}{\sqrt{5}\,F_{2m}} \ln\alpha,
\end{align*}
from which in view of the fact $L_{4m}-5F^2_{2m}=2$ we have
\begin{equation*}
\int_0^\infty \frac{x^2}{x^4+L_{4m}x^2+1}\ln x\,dx = - \frac{ m\pi}{\sqrt{5}F_{2m}} \ln\alpha.
\end{equation*}

\section{Concluding comments}

This work was inspired by the papers by Boyadzhiev  \cite{Boyadzhiev} and Chen \cite{Chen}.
Our first significant complement is a different proof of Chen's identity \eqref{chen_id} using a special family of log-integrals.
Based on this identity we derived several interesting combinatorial sum identities involving Catalan numbers, harmonic numbers,
and odd harmonic numbers. Proceeding further we offered generating functions for some related series and evaluated
them for specific Fibonacci-coefficients in closed form. Finally, we have discussed a family of related integrals
involving Fibonacci numbers and odd harmonic numbers.

We close this paper mentioning that more corollaries and further results can be obtained by writing \eqref{main_id1} as
\begin{equation*}
\sum_{n = 0}^\infty \binom{2n}{n} \frac{O_n}{4^n} x^n = -\frac{ \ln\sqrt{1-x}}{\sqrt{1-x}},\quad |x| < 1,
\end{equation*}
or
\begin{equation*}
\sum_{n = 0}^\infty \binom{2n}{n} \frac{O_n}{4^n} (1 - x^2)^n = -\frac{\ln x}{x},\quad |1 - x^2| < 1.
\end{equation*}
We immediately have the trigonometric versions
\begin{align*}
\sum_{n = 1}^\infty \binom{2n}{n} O_n \frac{{\cos ^n x}}{4^n}& =- \frac{\ln \big(\sqrt 2\, | {\sin (x/2)}|\big)}{\sqrt 2 \left| {\sin (x/2)} \right|} ,\\
\sum_{n = 1}^\infty (- 1)^{n-1} \binom{2n}{n} O_n \frac{\cos ^n x}{4^n} &= \frac{\ln \big(\sqrt 2\, |\cos (x/2)|\big)}{\sqrt 2| \cos (x/2)|},\\
\sum_{n = 0}^\infty \binom{2n}{n} O_n \frac{{\cos ^{2n} x}}{{4^n }} &= -\frac{\ln \left| {\sin x} \right| }{{\left| {\sin x} \right|}},
\end{align*}
and
\begin{align*}
\sum_{n = 0}^\infty \binom{2n}{n} O_n \frac{{\sin ^{2n} x}}{{4^n }} = - \frac{\ln \left| {\cos x} \right| }{{\left| {\cos x} \right|}}.
\end{align*}

By utilizing suitable values of $x$ from these formulas, we can derive several series that involve odd harmonic numbers. Here are some examples:
\begin{align*}
\sum_{n = 1}^\infty \binom{2n}{n} \frac{O_n}{8^n} &=\frac{\sqrt2}{2}\ln2,\\
\sum_{n = 1}^\infty \binom{2n}{n} \Big(\frac{3}{16}\Big)^n O_n&=2\ln2,\\
\sum_{n = 1}^\infty \binom{2n}{n} \Big(\frac{\sqrt5}{16\alpha}\Big)^n O_n &=-\frac{2}{\alpha}\ln\Big(\frac{\alpha}{2}\Big),\\
\sum_{n = 1}^\infty \binom{2n}{n} \Big(\frac{\sqrt5\alpha}{16}\Big)^n O_n  &={2}{\alpha}\ln(2{\alpha}).
\end{align*}


\newpage
\begin{thebibliography}{99}

\bibitem{Boyadzhiev}
K.~N. Boyadzhiev, \emph{Notes on the Binomial Transform: Theory and Table with Appendix on Stirling Transform}, World Scientific, 2018.
\vspace{-0.1cm}
\bibitem{CarFer}
L.~Carlitz and H.~H. Ferns, Some Fibonacci and Lucas identities, \textit{Fibonacci Quart.} {\bf8}(1) (1970), 61--73.



\bibitem{Benjamin}
A.~T. Benjamin, D. Gaebler and  R. Gaebler, A combinatorial approach to hyperharmonic numbers, \textit{Integers} \textbf{3} (2003), \#A15.

\bibitem{Boyadzhiev}
K.~N. Boyadzhiev, Series with central binomial coefficients, Catalan numbers, and harmonic numbers, \textit{J. Integer Seq.} \textbf{15} (2012), Article 12.1.7. 

\bibitem{Boyadzhiev2}
K.~N. Boyadzhiev and R. Frontczak, Hadamard product of series with special numbers, \textit{Funct. Approx. Comment. Math.} \textbf{68}(2) (2023), 231--247.  

\bibitem{Carlitz}
L. Carlitz, Solution to Problem H-285, \textit{Fibonacci Quart.} \textbf{18}(2) (1980),  191--192. 

\bibitem{Chen2}
H. Chen, Evaluation of some variant Euler sums, \textit{J. Integer Seq.} \textbf{9} (2006), Article 06.2.3. 

\bibitem{Chen}
H. Chen, Interesting series associated with central binomial coefficients, Catalan numbers and harmonic numbers, \textit{J. Integer Seq.} \textbf{19} (2016), Article 16.1.5.

\bibitem{Chu}
W. Chu and D. Zheng, Infinite series with harmonic numbers and central binomial coefficients,  \textit{Int. J. Number Theory} \textbf{5} (2009), 429--448.  

\bibitem{Dil}
A. Dil and K.\ N. Boyadzhiev, Euler sums of hyperharmonic numbers, \textit{J. Number Theory} \textbf{147} (2015), 490--498.

\bibitem{Furdui}
O. Furdui and A. S\^int\u am\u arian, Series involving products of odd harmonic numbers, \textit{Gaz. Mat. Ser. A} \textbf{3/4} (2022), 1--11.


\bibitem{GrRy07}
I. Gradshteyn and I. Ryzhik, \textit{Table of Integrals, Series, and Products}, Elsevier Academic Press, 2007.

\bibitem{Lehmer}
D.\ H. Lehmer, Interesting series involving the central binomial coefficient, \textit{Amer. Math. Monthly} \textbf{92} (1985), 449--457.

\bibitem{Furdui2}
A. S\^int\u am\u arian and O. Furdui, \textit{Sharpening Mathematical Analysis Skills}, Springer, 2021.

\bibitem{OEIS}
N.\ J.\ A. Sloane (ed.), The On-Line Encyclopedia of Integer Sequences, 2020. Available at https://oeis.org.

\bibitem{Spiess}
J. Spie\ss, Some identities involving harmonic numbers, \textit{Math. Comp.} \textbf{55}(192) (1990),  839--863. 


\bibitem{Stewart}
S.~M. Stewart, Some series involving products between the harmonic numbers and the Fibonacci numbers, \textit{Fibonacci Quart.} \textbf{59}(3) (2021),  214--224.









\end{thebibliography}
\end{document}